\documentclass{amsart}
\usepackage{amsmath}
\usepackage{amssymb}
\usepackage{amsthm}
\usepackage{enumerate}
\usepackage[dvips]{graphicx}
\usepackage{float, pict2e, bxeepic}

\usepackage[usenames,dvipsnames]{xcolor}

\usepackage{capt-of}
\usepackage{tikz}
\usetikzlibrary{patterns}
\usetikzlibrary{decorations.markings}
\usepackage{multirow}
\usepackage{array}

\newtheorem{theorem}{Theorem}
\newtheorem{prop}[theorem]{Proposition}
\newtheorem{lemma}[theorem]{Lemma}
\newtheorem{rem}[theorem]{Remark}
\newtheorem{exmp}[theorem]{Example}

\begin{document}
\title{Maximal cliques in the graph of $5$-ary simplex codes of dimension two}
\author{Mariusz Kwiatkowski, Andrzej Matra\'s, Mark Pankov, Adam Tyc}
\subjclass[2020]{51E20, 51E22.}
\address{Faculty of Mathematics and Computer Science, 
University of Warmia and Mazury, S{\l}oneczna 54, 10-710 Olsztyn, Poland}
\email{mkw@matman.uwm.edu.pl, matras@uwm.edu.pl, pankov@matman.uwm.edu.pl \\ adam.tyc@matman.uwm.edu.pl}

\maketitle

\begin{abstract}
We consider  the induced subgraph of the corresponding Grassmann graph formed by $q$-ary simplex codes of dimension $2$, $q\ge 5$. 
This graph contains precisely two types of maximal cliques.
If $q=5$, then for any two maximal cliques of the same type there is a monomial linear automorphism transferring one of them to the other.
Examples concerning the cases $q=7,11$ finish the note.
\end{abstract}

\section{Introduction}
The Grassmann graph formed by $k$-dimensional subspaces of an $n$-dimensional vector space over the $q$-element field 
(see, for example, \cite[Section 9.3]{BCN}) can be naturally identified with the graph of $q$-ary linear codes of dimension $k$ and length $n$,
where two distinct codes are connected by an edge if they have the maximal number of common codewords
(for $k=1,n-1$ this graph is complete and we assume that $1<k<n-1$). 
In practice, only non-degenerate linear codes are interesting, since degenerate codes can be reduced to 
non-degenerate codes of shorter length by removing zero columns from generator matrices.
The induced subgraphs of Grassmann graphs consisted of all non-degenerate linear codes 
and subgraphs formed by various classes of non-degenerate linear codes (projective and simplex codes, codes with lower bounded
minimal dual distance) are considered in 
\cite{BM,KP1,KP2,P1,P2} and \cite{CGK,CG,KPP, KP2020,KP2023}, respectively.

In this note, we continue the investigation of maximal cliques (maximal complete subgraphs) in the graphs of simplex codes  started at \cite{KP2023}.

Every maximal clique of the Grassmann graph is 
a star (formed by all $k$-dimen\-sional subspaces containing a fixed $(k-1)$-dimensional subspace)
or a top (consisting of all $k$-dimensional subspaces contained in a certain $(k+1)$-dimensional subspace).
If $\Gamma$ is an induced subgraph of the Grassmann graph, then the intersection of the vertex set of $\Gamma$
with a star or a top is a not necessarily  maximal clique of $\Gamma$ (there are examples of such non-maximal cliques).
This clique  is called a star or, respectively, a top of $\Gamma$ only in the case when it is maximal.

Maximal cliques in the graphs of non-degenerate codes are studied in \cite{KP2} and later in \cite{BM}. 
Some of their properties are used to determining graph automorphisms \cite{KP2}
and embeddings in Grassmann graphs \cite{P1,P2}.

The graph $\Gamma^s(k,q)$ formed by $q$-ary simplex codes of dimension $k$ 
contains stars and for any two stars there is a monomial linear automorphism transferring one of them to the other. 
All maximal cliques of $\Gamma^s(3,2)$ and $\Gamma^s(2,3)$ are stars and tops simultaneously
and  every maximal clique of $\Gamma^s(2,4)$ is a star. 
By \cite{KP2023}, the graph $\Gamma^s(k,q)$ contains tops for each of the remaining cases;
furthermore, there are tops with distinct numbers of elements if $k\ge 4$ and $q\ge 3$.

We use a point-line geometry approach to study $q$-ary simplex codes of dimension $2$
and consider the subgeometry of the corresponding projective space formed by all simplex points,
i.e. $1$-dimensional subspaces containing codewords of simplex codes, and simplex lines corresponding to simplex codes.
The inversion transformation $I$ does not send lines to lines; however, it preserves the binary collinearity relation of simplex points.
We exploit this observation  to construct tops in $\Gamma^s(2,q)$, $q\ge 5$
(the arguments from \cite{KP2023} are more complicated).

Our main result concerns the graph of $5$-ary simplex codes of dimension $2$. We show that every top of this graph contains precisely $4$ elements and for any two tops there is a monomial linear automorphism transferring one of them to the other.
The crucial tool is the following characterization of the collinearity relation based on the inversion transformation $I$ for the case $q=5$:
if $P,Q$ are collinear simplex points and $\langle P,Q\rangle$ is the simplex line containing them, 
then $$I(\langle I(P),I(Q)\rangle)\setminus\{P,Q\}$$
is the set of all simplex points which are collinear to both $P,Q$ and not on the line $\langle P,Q\rangle$. 
For $q=7,11$ this fails (see examples at the end of the note).
Determining of the set of all simplex points collinear to both $P,Q$ in the general case is an open problem.

\section{Simplex codes}
Let ${\mathbb F}={\mathbb F}_q$ be the field of $q$ elements.
Consider the $n$-dimensional vector space $V={\mathbb F}^n$ over this field.
Subspaces of $V$ are $q$-{\it ary linear codes} and vectors belonging to a linear code $C$ are {\it codewords} of $C$.
Recall that a {\it generator matrix} of a linear code is any matrix whose rows form a basis of this code.

Suppose that 
\begin{equation}\label{eq0}
n=\frac{q^{k}-1}{q-1},
\end{equation}
i.e. the dimension of $V$ is equal to the number of points in a $(k-1)$-dimensional projective space over ${\mathbb F}$
(the number of $1$-dimensional subspaces of ${\mathbb F}^k$).
A $q$-{\it ary simplex code of dimension} $k$ is a $k$-dimensional linear code $C\subset V$ such that 
every generator matrix of $C$ has the following property:
\begin{enumerate}
\item[(*)] all columns are non-zero and mutually non-proportional.
\end{enumerate}
The condition (*) together with \eqref{eq0} imply  the existence of  
a one-to-one correspondence between matrix columns and points of a $(k-1)$-dimensional projective space over ${\mathbb F}$.
If (*) is satisfied for one generator matrix of $C$, then it holds for all generator matrix of this code.

Non-zero codewords of  $q$-ary simplex codes of dimension $k$  are of Hamming weight $q^{k-1}$  \cite[Theorem 2.7.5]{HP}.
Furthermore, such codes can be characterized as linear codes $C\subset V$ maximal with respect to this property 
(this is a simple consequence of a well-known fact concerning  equidistant codes \cite[Theorem 7.9.5]{HP}).

A linear automorphism of $V$ is {\it monomial} if there exist non-zero $a_1,\dots,a_n\in {\mathbb F}$
and a permutation $\sigma$ on $\{1,\dots,n\}$ such that 
this automorphism transfers every $(x_1,\dots,x_n)$ to 
$$(a_1x_{\sigma(1)},\dots,a_nx_{\sigma(n)}).$$
Such linear automorphisms preserve the Hamming weight and, consequently, they send simplex codes to simplex codes.
It is clear that for any two vectors of $V$ with the same Hamming weight there is a monomial linear automorphism of $V$
transferring one of these vectors to the other. 
This implies that every vector of $V$ with Hamming weight $q^{k-1}$ is a codeword in a certain simplex code. 
Furthermore, for any two $q$-ary simplex codes of dimension $k$ there is a monomial linear automorphism of $V$
transferring one of these codes to the other. 
The latter easy follows from MacWillams theorem which states that every linear isomorphism between two linear codes of $V$
preserving the Hamming weight  can be extended to a monomial linear isomorphism of $V$.

From this moment, we assume that $k=2$. 
Then \eqref{eq0} implies that $n=q+1$.

We use the point-line geometry language.
Let  ${\mathcal P}(V)$ be the projective space associated to $V$.
Recall that points of ${\mathcal P}(V)$ are $1$-dimensional subspaces $\langle x \rangle$, $x\in V\setminus\{0\}$ and lines 
correspond to $2$-dimensional subspaces of $V$. The line containing distinct points $P,Q\in {\mathcal P}(V)$ is denoted by 
$\langle P,Q\rangle$.
Lines of ${\mathcal P}(V)$ corresponding to $q$-ary simplex codes of dimension $2$ will be called {\it simplex lines} 
and points belonging to such lines are said to be {\it simplex points}.
A point of ${\mathcal P}(V)$ (a $1$-dimensional subspace of $V$) is simplex if and only if 
it contains a vector whose coordinates are non-zero except one
(a non-zero codeword of a certain $q$-ary simplex code of dimension $2$).

Consider the subgeometry of ${\mathcal P}(V)$ formed by all simplex points and all simplex lines.
Two distinct simplex points $\langle x\rangle, \langle y \rangle$ are collinear in this geometry, i.e. there is a simplex line containing them,
if and only if the columns of  the generator matrix $\genfrac[]{0pt}{2}{x}{y}$ are non-zero and mutually non-proportional.
In what follows, we will say that two distinct simplex points are {\it collinear} if they are collinear in the above geometry.

\section{Main result}
Let $\Gamma_2(V)$ be the Grassmann graph whose vertices are lines of  ${\mathcal P}(V)$
($2$-dimensional subspaces of $V$) and two lines are connected by an edge if they are {\it adjacent}, i.e.  intersecting in a point
(a $1$-dimensional subspace of $V$). 
This is a connected graph with maximal distance  $2$
(for any two disjoint lines there is a line adjacent to each of them). 
Every maximal clique of $\Gamma_2(V)$ is of one of the following two types:
\begin{enumerate}
\item[$\bullet$] the {\it star} formed by all lines passing through a fixed point;
\item[$\bullet$] the {\it top} consisting of all lines contained in a fixed projective plane.
\end{enumerate}

Denote by $\Gamma$ the subgraph of $\Gamma_2(V)$ induced by the set of all simplex lines 
(in \cite{KP2023} the graph of $q$-ary simplex codes of dimension $k$ is denoted by $\Gamma^s(k,q)$ and, consequently, 
$\Gamma=\Gamma^s(2,q)$). 
This graph is connected \cite[Proposition 4]{KP2023}. 
The intersection of the set of simplex lines with a star or a top is a clique of $\Gamma$ (if it is not empty);
such a clique need not to be maximal. We say that this intersection is a {\it star} of $\Gamma$ or, respectively, a {\it top} of $\Gamma$
only when it is a maximal clique of $\Gamma$. 

\begin{rem}{\rm
By \cite[Example 7]{KP2023}, 
for every $q\ge 4$  there are tops of $\Gamma_2(V)$ intersecting the set of simplex lines in non-maximal cliques of $\Gamma$
(every such intersection is a proper subset of a certain star of $\Gamma$).
}\end{rem}

It is clear that a point of ${\mathcal P}(V)$ defines a star of $\Gamma$ if and only if it is a simplex point.
Therefore, for any two stars of $\Gamma$ there is a monomial linear automorphism of $V$ 
transferring one of these stars to the other.
Every star of $\Gamma$ consists of  $(q -1)!$ lines \cite[Proposition 1]{KP2020}.  
If $q=4$, then every maximal clique of $\Gamma$ is a star \cite[Proposition 3]{KP2020}. 

We present a new proof of the following statement obtained in \cite{KP2023}.

\begin{theorem}\label{theorem-top-gen}
For every $q\ge 5$ the graph $\Gamma$ contains tops.
\end{theorem}

Our main result is the following.

\begin{theorem}\label{theorem-top-q5}
Suppose that $q=5$. 
Then every top of $\Gamma$ contains precisely $4$  lines.
For any two tops there is a monomial linear automorphism of $V$ 
transferring one of these tops to the other. 
Every line is contained in precisely $20$ tops.
\end{theorem}

\section{Proof of Theorems \ref{theorem-top-gen} and \ref{theorem-top-q5}}

\subsection{The inversion transformation of ${\mathcal P}(V)$}
Let $I$ be the inversion transformation of ${\mathbb F}$ which sends every 
$a\in {\mathbb F}\setminus\{0\}$ to $a^{-1}$ and leaves $0$ fixed. 
For $q=2,3$ this transformation is identity. 
If $q=4$, then it is the automorphism of ${\mathbb F}$ transferring every $a$ to $a^2$
(we have $(a+b)^2=a^2+b^2$ in the case of characteristic $2$).

The transformation of $V$ sending every $(x_1,\dots,x_n)$ to $(I(x_1),\dots,I(x_n))$ will be denoted by the same symbol $I$.
For $q=2,3$ this is the identity. 
If $q=4$, then we obtain the monomial semilinear automorphism of $V$ associated to the field automorphism $a\to a^2$
and leaving fixed all vectors of the standard basis.
We have $I(ax)=I(a)I(x)$ for every $a\in {\mathbb F}$ and $x\in V$ which means that
$I$ induces a bijective transformation of ${\mathcal P}(V)$ preserving the set of simplex points.

If $x=(x_1,\dots, x_n)$ and $y=(y_1,\dots, y_n)$, then we define
$$x\cdot y=x_1y_1+\dots +x_ny_n.$$
For a point $P=\langle x\rangle$ we denote by $H(P)$ the hyperplane formed by all points $\langle y \rangle$
satisfying $I(x)\cdot y=0$. 
The hyperplanes $H(P)$ and $H(Q)$ are distinct for distinct points $P,Q$.

\begin{lemma}\label{lemmaH}
For every simplex point $P$ the following assertions are fulfilled:
\begin{enumerate}
\item[{\rm(1)}] $P$ belongs to $H(P)$;
\item[{\rm(2)}] every simplex point collinear to $P$ belongs to $H(P)$.
\end{enumerate}
\end{lemma}

\begin{proof}
(1). If $P=\langle x\rangle$ and $x=(x_1,\dots, x_n)$, then only one of $I(x_i)x_i$ is zero 
and  each of the remaining $I(x_i)x_i$ is $1$ which means that 
$$I(x)\cdot x=\underbrace{1+\dots+1}_{q}=0.$$
(2). For a simplex point $Q=\langle y\rangle$, $y=(y_1,\dots, y_n)$
consider the $(2\times n)$-matrix whose rows are $x$ and $y$.
If $Q$ is collinear to $P$, then all columns of this matrix are non-zero and mutually non-proportional. 
The latter implies that precisely two of $I(x_i)y_i$ are zero 
and the remaining $I(x_i)y_i$ are mutually distinct; in particular,
for every non-zero $a\in {\mathbb F}$ there is a unique index $i$ such that $I(x_i)y_i=a$.
Since the sum of all non-zero elements of ${\mathbb F}$ is zero, we obtain that $I(x)\cdot y=0$. 
\end{proof}

\begin{rem}{\rm
In the case when $q=4$, a simplex point $Q$ is collinear to a simplex point $P$  if and only if $Q$ belongs to $H(P)$
(this follows from the fact that three non-zero elements of ${\mathbb F}_4$ are mutually distinct if and only if 
their sum is zero). 
Show that for every $q\ge 5$ the hyperplane $H(P)$ contains simplex points non-collinear to $P$.
If $q\ge 5$, then there are non-zero $a_1,\dots,a_{q-1}\in {\mathbb F}$ such that
$$a_1+\dots+a_{q-1}=0$$
and $a_i=a_j$ for some distinct $i,j$
(see \cite[Remark 1]{KP2020} for the details).
If 
$$P=\langle (0,1,\dots,1)\rangle\;\mbox{ and }\; Q=\langle (1,0, a_1,\dots,a_{q-1})\rangle,$$
then $Q$ is a simplex point of $H(P)$ non-collinear to $P$.
For every simplex point $P'$
there is a monomial linear automorphism of $V$ sending $P$ to $P'$.
It transfers $H(P)$ to $H(P')$ and $Q$ to a simplex point of $H(P')$ non-collinear to $P'$.
}\end{rem}

\begin{prop}\label{propF}
The following assertions are fulfilled:
\begin{enumerate}
\item[{\rm (1)}] Simplex points $P,Q$ are collinear if and only if $I(P),I(Q)$ are collinear.
\item[{\rm (2)}] If $q\ge 5$, then for any mutually distinct points $P_1,P_2,P_3$ on a simplex line
the points $I(P_1),I(P_2),I(P_3)$ span a projective plane. 
 \end{enumerate}
\end{prop}

To prove the statement (2) from Proposition \ref{propF} we need the following lemmas.

\begin{lemma}\label{lemma3p}
For any mutually distinct points $P_1,P_2,P_3$ and $P'_1,P'_2,P'_3$ on simplex lines $L$ and $L'$, respectively,
there is a monomial linear automorphism $l$ of $V$ such that $l(P_i)=P'_i$ for every $i\in \{1,2,3\}$.
\end{lemma}

\begin{proof}
Let $C$ and $C'$ be the simplex codes corresponding to $L$ and $L'$, respectively. 
There is a linear isomorphism $l:C\to C'$ satisfying $l(P_i)=P'_i$ for every $i\in \{1,2,3\}$.
By MacWillams theorem, $l$ can be extended to a monomial linear automorphism of $V$.
\end{proof}

\begin{lemma}\label{lemmaF}
For every monomial linear automorphism $l$ of $V$ there is a monomial linear automorphism $l'$ of $V$ such that
$$Il=l'I.$$
\end{lemma}

\begin{proof}
If $l$ transfers every $(x_1,\dots,x_n)$ to
$$(a_1x_{\sigma(1)},\dots,a_nx_{\sigma(n)}),$$
then the monomial linear automorphism 
$$l':(x_1,\dots,x_n)\to (a^{-1}_1x_{\sigma(1)},\dots,a^{-1}_nx_{\sigma(n)})$$
is as required.
\end{proof}

\begin{proof}[Proof of Proposition \ref{propF}]
(1). 
Simplex points $P=\langle x\rangle, Q=\langle y\rangle$ are collinear if and only if precisely two of $I(x_i)y_i$ are zero 
and the remaining $I(x_i)y_i$ are mutually distinct.
Since $I$ is bijective and $$I(I(x_i)y_i)=I(I(x_i))I(y_i),$$
the latter holds if and only if precisely two of $I(I(x_i))I(y_i)$ are zero 
and the remaining  $I(I(x_i))I(y_i)$ are mutually distinct which is equivalent to the fact that the simplex points $I(P),I(Q)$ are collinear.

(2). By Lemmas \ref{lemma3p} and \ref{lemmaF}, it is sufficient to check the statement for the simplex points $\langle x \rangle, \langle y \rangle, \langle x+y \rangle$,
where
$$x=(0,1,\dots,1),\; y=(1,0,1,\alpha,\dots,\alpha^{q-2}),\;x+y=(1,1,2,\alpha+1,\dots,\alpha^{q-2}+1)$$
and $\alpha$ is a primitive element of the field. 

Suppose that the field characteristic is not $2$ or $3$ and consider
the vectors
$$(0,1,1),\; (1,0,1),\; (1,1,2^{-1})$$
formed by the first three coordinates of $I(x),I(y),I(x+y)$, respectively.
Since $2\ne 2^{-1}$ (the characteristic of $\mathbb{F}$ is not $2$ or $3$),
these vectors are linearly independent and, consequently, $I(x),I(y),I(x+y)$ are linearly independent. 

Now, we assume that the characteristic is $2$ or $3$ and $q\neq 2,3,4$.
Consider the vectors
$$(0,1,1),\; (1,0,\alpha^{-1}),\; (1,1,(\alpha +1)^{-1})$$
formed by the first, second and fourth coordinates of $I(x),I(y),I(x+y)$, respectively. 
If these vectors are linearly dependent, then $1+\alpha^{-1}=(1+\alpha)^{-1}$ and
$$(1+\alpha)(1+\alpha^{-1})=1,$$
$$\alpha+\alpha^{-1}+2=1,$$
$$\alpha+1+\alpha^{-1}=0,$$
$$\alpha^2+\alpha+1=0.$$
If the characteristic is $2$, then the latter equality  holds only  when $\alpha$ is the primitive element of $\mathbb{F}_4$.
If the characteristic  is $3$, then the polynomial 
$$x^2+x+1=x^2-2x+1=(x-1)^2$$
is reducible and there is no primitive element $\alpha\in \mathbb{F}_{3^m}$, $m\geq2$ satisfing
$\alpha^2+\alpha+1=0$.
So, $I(x),I(y),I(x+y)$ are linearly independent for each of these cases.
\end{proof}

\begin{rem}{\rm
For $q=2,3,4$ the statement (2) from Proposition \ref{propF} fails (in these cases, $I$ is identity or a monomial semilinear automorphism).
}\end{rem}

We will need the following statement concerning the case $q=5$.

\begin{lemma}\label{lemmaF5}
 If $q=5$, then  for any mutually distinct points $P_1,P_2,P_3,P_4$ on a simplex line
the points $I(P_1),I(P_2),I(P_3),I(P_4)$ span a $3$-dimensional projective space.
\end{lemma}

\begin{proof}
By Lemmas \ref{lemma3p} and \ref{lemmaF}, it is sufficient to consider the case when 
$$P_1=\langle x \rangle,\;P_2=\langle y \rangle,\;P_3=\langle x+y \rangle,$$
where
$$x=(0,1,1,1,1,1),\;y=(1,0,1,2,4,3),\;x+y=(1,1,2,3,0,4),$$
and $P_4=\langle ax+y\rangle$ with $a\in \{2,3,4\}$.
We have
$$I(x)=(0,1,1,1,1,1),\;I(y)=(1,0,1,3,4,2),\;I(x+y)=(1,1,3,2,0,4),$$
$$I(2x+y)=(1,3,2,4,1,0),\;I(3x+y)=(1,2,4,0,3,1),\; I(4x+y)=(1,4,0,1,2,3).$$
A direct calculation shows that for each $a\in \{2,3,4\}$ the determinant whose rows are formed by the first four coordinates of
$I(x),I(y),I(x+y),I(ax+y)$ is non-zero.
\end{proof}

\begin{rem}\label{rem-con}{\rm
Our conjecture is the following:  if $q\ge 5$, then $I$ transfers any $n-2=q-1$ mutually distinct points on a simplex line to 
points spanning an $(n-3)$-dimensional projective space.
For example, if $q=7$, then it is sufficient to show that the vectors 
$$I(x),\; I(y),\; I(x+y),\;I(ax+y),\;I(bx+y),\;I(cx+y),$$
where $x=(0,1,1,1,1,1,1,1)$, $y=(1,0,1,3,2,6,4,5)$, are linearly independent
for any  mutually distinct $a,b,c\in {\mathbb F}\setminus\{0,1\}$. 
There are precisely $\binom{5}{3}=10$ such triples $a,b,c$.
Calculations show that  the $10$ determinants  whose rows are formed by the first six coordinates of these vectors are non-zero.
This supports the conjecture. 
}\end{rem}

\subsection{Proof of Theorem \ref{theorem-top-gen}}
Suppose that $q\ge 5$. 
Let $P_1,P_2,P_3$ be mutually distinct points on a simplex line. 
By the statement (1) from Proposition \ref{propF}, 
the simplex points $I(P_1),I(P_2),I(P_3)$ are mutually collinear. 
The statement (2) from Proposition \ref{propF} says that these points span a projective plane 
and, consequently, the three simplex lines joining pairs of these points belong to a certain top of $\Gamma$.

\begin{rem}{\rm
In \cite{KP2023}, a top of $\Gamma$ is constructed without the inversion transformation $I$.
The arguments from \cite{KP2023} are more complicated.
}\end{rem}

\subsection{Proof of Theorem \ref{theorem-top-q5}}
Suppose that $q=5$. Then $n=6$.

Let $P$ and $Q$ be collinear simplex points. By Proposition \ref{propF},
the simplex points $I(P)$ and $I(Q)$ are collinear.
Since $I$ is an involution, 
$I(\langle I(P),I(Q)\rangle)$ contains $P$ and $Q$. 
It follows from Proposition \ref{propF}  that the remaining $4$ points from this set are not on the line $\langle P,Q\rangle$.
Also, Proposition \ref{propF} shows that the points from $I(\langle I(P),I(Q)\rangle)$ are mutually collinear.

\begin{lemma}\label{lemma-4p}
For any collinear simplex points $P,Q$ there are precisely $4$ simplex points which are collinear to both $P,Q$ and do not belong to the line $\langle P,Q\rangle$.
\end{lemma}

\begin{proof}
Since any pair of collinear simplex points can be transferred to 
any other pair of such points by a monomial linear automorphism of $V$,
it is sufficient to determine all simplex points
collinear to both 
$$P_1=\langle0,1,1,1,1,1\rangle,\;P_2=\langle1,0,1,2,4,3\rangle.$$
The remaining four points on the simplex line $\langle P_1, P_2\rangle$ are 
$$P_3=\langle1,4,0,1,3,2\rangle,$$
$$P_4=\langle1,3,4,0,2,1\rangle,$$
$$P_5=\langle1,1,2,3,0,4\rangle,$$
$$P_6=\langle1,2,3,4,1,0\rangle.$$
The simplex points 
\begin{equation}\label{eqP1P2}
I(P_1)=P_1=\langle0,1,1,1,1,1\rangle\;\mbox{ and }\; I(P_2)=\langle1,0,1,3,4,2\rangle
\end{equation}
are collinear and the remaining four points on the line joining them are
$$Q'_3=\langle1,4,0,2,3,1\rangle,$$
$$Q'_4=\langle1,2,3,0,1,4\rangle,$$
$$Q'_5=\langle1,1,2,4,0,3\rangle,$$
$$Q'_6=\langle1,3,4,1,2,0\rangle.$$
The simplex points 
$$Q_3=I(Q'_3)=\langle1,4,0,3,2,1\rangle,$$
$$Q_4=I(Q'_4)=\langle1,3,2,0,1,4\rangle,$$
$$Q_5=I(Q'_5)=\langle1,1,3,4,0,2\rangle,$$
$$Q_6=I(Q'_6)=\langle1,2,4,1,3,0\rangle$$
are collinear to both $P_1,P_2$ and each of them is not on the line $\langle P_1,P_2\rangle$. 

Let   $T=\langle x_1,\dots,x_6\rangle$ be a simplex point collinear to both $P_1,P_2$.
Lemma \ref{lemmaH} shows that 
$T$ belongs to the intersection of the hyperplanes $H(P_1)$ and $H(P_2)$.
Then, by \eqref{eqP1P2}, we have
\begin{equation}\label{eq1}
x_2+x_3+x_4+x_5+x_6=0\;\mbox{ and }\;x_1+x_3+3x_4+4x_5+2x_6=0.
\end{equation}
If $T$ is collinear to a simplex point $\langle y_1,\dots,y_6\rangle$, then $x_i=0$ implies that $y_i\ne 0$.
Since $T$ is collinear to $P_1=\langle0,1,1,1,1,1\rangle,P_2=\langle1,0,1,2,4,3\rangle$, we obtain that 
$$x_i=0\;\mbox{ for a unique }\; i\ge 3.$$  
Also, we can assume that $x_1=1$.

Consider the case when $x_3=0$.
Then \eqref{eq1} shows that
\begin{equation}\label{eq2}
x_2=2x_5+3x_6+2\;\mbox{ and }\;x_4=2x_5+x_6+3.
\end{equation}
Recall that simplex points $\langle x\rangle, \langle y \rangle$ are collinear if and only if 
the columns of the $(2\times 6)$-matrix $\genfrac[]{0pt}{2}{x}{y}$ are mutually non-proportional.
Since 
$$T=\langle 1,x_2,0,x_4,x_5,x_6\rangle,\; P_1=\langle0,1,1,1,1,1\rangle$$ 
are collinear, $x_2,x_4,x_5,x_6$ are non-zero and mutually distinct.
The collinearity of 
$$T=\langle 1,x_2,0,x_4,x_5,x_6\rangle,\;P_2=\langle1,0,1,2,4,3\rangle$$ implies that $1,2^{-1}x_4,4^{-1}x_5,3^{-1}x_6$ are mutually distinct 
which means that 
$$x_5\neq4,\;x_6\neq3,\; x_5\ne 3x_6.$$ 
Consider all possibilities for $x_6$:
\begin{enumerate}
\item[$\bullet$] If $x_6=1$, then $x_5\neq 1,3,4$ and, consequently, $x_5=2$ which means that
$T=\langle 1,4,0,3,2,1\rangle=Q_3$.
\item[$\bullet$]
If $x_6=2$, then $x_5\neq 1,2,4$. Thus $x_5=3$ and
$T=\langle 1,4,0,1,3,2\rangle=P_3$.
\item[$\bullet$]
It was noted above that $x_6\ne 3$.
\item[$\bullet$]
If $x_6=4$,  then $x_5\neq 2,4$. 
Since $x_2\ne x_5$,  we have 
$$0\neq x_2-x_5=x_5+3x_6+2=x_5+4$$
by \eqref{eq2} and, consequently,  $x_5\neq 1$. 
Similarly, $x_4\ne x_5$ and \eqref{eq2} show that 
$$0\neq x_4-x_5=x_5+x_6+3=x_5+2$$
and $x_5\neq 3$. 
Therefore, the possibility $x_6=4$ is not realized.
\end{enumerate}
So, there are precisely two simplex points $\langle 1,\cdot,0,\cdot,\cdot,\cdot \rangle$ collinear to both $P_1,P_2$;
these points are $P_3$ and $Q_3$.

The monomial linear automorphism $l$ of $V$ transferring every $(t_1,\dots,t_6)$ to
$$(t_1,2t_2,2t_6,2t_3,2t_4,2t_5)$$
leaves fixed both $P_1,P_2$ and 
induces the permutations
$$(P_3,P_4,P_5,P_6)\;\mbox{ and }\;(Q_3,Q_4,Q_5,Q_6).$$
As above,  $T=\langle1,x_2,\dots,x_6\rangle$ is a simplex point collinear to both $P_1,P_2$. 
If $x_i=0$ for a certain $i\ge 4$, then 
there is a natural $k$ such that $$l^k(T)=\langle 1,\cdot,0,\cdot,\cdot,\cdot \rangle.$$
Then $l^k(T)$ is $P_3$ or $Q_3$ and, consequently, $T$ is $P_i$ or $Q_i$ for some $i\in \{4,5,6\}$.
\end{proof}

So, if $P,Q$ are collinear simplex points, then 
\begin{equation}\label{eq3}
I(\langle I(P),I(Q)\rangle)\setminus \{P,Q\}
\end{equation}
is the set of all simplex points which are collinear to both $P,Q$ and do not belong to the  line $\langle P,Q\rangle$.

\begin{lemma}\label{lemma-4p2}
For any collinear simplex points $P,Q$ every point from the set \eqref{eq3}
is not collinear to points of the line $\langle P,Q\rangle$ distinct from $P,Q$. 
\end{lemma}

\begin{proof}
Let $T$ be a point on the simplex line $\langle P,Q\rangle$ distinct from $P,Q$. 
By the statement (2) of Proposition \ref{propF},
$I(P),I(Q),I(T)$ span a projective plane and, consequently,
$$\langle I(P),I(Q)\rangle\cap \langle I(P),I(T)\rangle=\{I(P)\}.$$
Since $I$ is an involution, we have
$$I(\langle I(P),I(Q))\rangle\cap I(\langle I(P),I(T)\rangle)=\{P\}$$
which means that the set \eqref{eq3} does not contain points collinear to $T$.
\end{proof}

\begin{prop}\label{prop-top1}
Every top of $\Gamma$ contains precisely $4$  lines.
\end{prop}

\begin{proof}
Every top of $\Gamma$ contains three lines which pairwise intersect in $3$ mutually distinct simplex points $P,Q,T$. 
These points span a projective plane.  Since $T$ is collinear to both $P,Q$, it belongs to $I(\langle I(P),I(Q)\rangle)$.
Let $P',Q',T'$ be the points of $I(\langle I(P),I(Q)\rangle)$ distinct from $P,Q,T$. 
Recall that the points from $I(\langle I(P),I(Q)\rangle)$ are mutually collinear. 
Lemma \ref{lemmaH} shows that every point from $I(\langle I(P),I(Q)\rangle)$ is contained in 
the intersection of the hyperplanes $H(P)$ and $H(Q)$.
By Lemma \ref{lemmaF5},
the projective planes  $$\langle P,Q,T\rangle\;\mbox{ and }\; \langle P',Q',T'\rangle$$
are distinct. 
These planes are contained in the $3$-dimensional projective space $H(P)\cap H(Q)$ 
and, consequently, their intersection is a certain (not necessarily simplex) line $L$. 
Since each of the points $P,Q,T$ doest not belong to $\langle P',Q',T'\rangle$ (Lemma \ref{lemmaF5}),
$L$ intersects the lines
\begin{equation}\label{eq4}
\langle P,Q\rangle,\;\langle P,T\rangle,\;\langle Q,T\rangle
\end{equation}
in three mutually distinct points $A,B,C$ (respectively) such that
$$\{P,Q,T\}\cap \{A,B,C\}=\emptyset,$$
see Fig.1(a).
Similarly, $L$ intersects the lines 
\begin{equation}\label{eq5}
\langle P',Q'\rangle,\;\langle P',T'\rangle,\;\langle Q',T'\rangle,
\end{equation}
in three mutually distinct points $A',B',C'$ (respectively) and
$$\{P',Q',T'\}\cap \{A',B',C'\}=\emptyset,$$
see Fig.1(b).
It follows from Lemma \ref{lemmaF5} that each of the lines \eqref{eq4} does not intersect the lines \eqref{eq5} which means that 
$$\{A,B,C\}\cap \{A',B',C'\}=\emptyset.$$
Therefore, the line $L$ is formed by the $6$ points $A,B,C,A',B',C'$. 
Each of these points is simplex and, consequently, $L$ is a simplex line belonging to our top.

\begin{center}
\begin{tikzpicture} 

\begin{scope}
\draw[fill=black] (90:2cm) circle (2pt);
\draw[fill=black] (-30:2cm) circle (2pt);
\draw[fill=black] (210:2cm) circle (2pt);
\draw (0,0) circle (1cm);

\draw (90:2cm)--(-30:2cm)--(210:2cm)--cycle;

\draw[fill=black] (-90:1cm) circle (1.5pt);
\draw[fill=black] (30:1cm) circle (1.5pt);
\draw[fill=black] (150:1cm) circle (1.5pt);

\draw[fill=white] (-30:1cm) circle (1.5pt);
\draw[fill=white] (90:1cm) circle (1.5pt);
\draw[fill=white] (210:1cm) circle (1.5pt);

\node at (-2.1cm,-1cm) {$P$};
\node at (2.1cm,-1cm) {$Q$};
\node at (0,-1.25cm) {$A$};

\node at (90:2.3cm) {$T$};

\node at (150:1.35cm) {$B$};
\node at (30:1.35cm) {$C$};

\node at (0:-0.8cm) {$L$};

\node at (0,-2cm) {(a)};
\end{scope}

\begin{scope}[xshift=6cm]
\draw[fill=black] (90:2cm) circle (2pt);
\draw[fill=black] (-30:2cm) circle (2pt);
\draw[fill=black] (210:2cm) circle (2pt);
\draw (0,0) circle (1cm);

\draw (90:2cm)--(-30:2cm)--(210:2cm)--cycle;

\draw[fill=white] (-90:1cm) circle (1.5pt);
\draw[fill=white] (30:1cm) circle (1.5pt);
\draw[fill=white] (150:1cm) circle (1.5pt);

\draw[fill=black] (-30:1cm) circle (1.5pt);
\draw[fill=black] (90:1cm) circle (1.5pt);
\draw[fill=black] (210:1cm) circle (1.5pt);

\node at (-2.1cm,-1cm) {$P'$};
\node at (2.1cm,-1cm) {$Q'$};
\node at (0,-1.25cm) {$A'$};

\node at (90:2.3cm) {$T'$};

\node at (150:1.35cm) {$B'$};
\node at (30:1.35cm) {$C'$};

\node at (0:-0.8cm) {$L$};

\node at (0,-2cm) {(b)};
\end{scope}

\end{tikzpicture}
\captionof{figure}{ }
\end{center}

Let $L'$ be a line belonging to the top.
This line intersects $\langle P,Q\rangle$ in a certain point $X$. 
Recall that $L$ intersects $\langle P,Q\rangle$ in $A$.
We assert that the plane $\langle P,Q,T\rangle$ is contained in
\begin{equation}\label{eq6}
H(P)\cap H(Q)\cap H(A)\cap H(X).
\end{equation}
It was noted above that  $\langle P,Q,T\rangle$ is contained in $H(P)\cap H(Q)$.
The lines $\langle P,Q\rangle$ and $L$ intersect in the simplex point  $A$ and, consequently,
these lines are contained in $H(A)$ by Lemma \ref{lemmaH}. 
Since the plane $\langle P,Q,T\rangle$ is spanned by $\langle P,Q\rangle$ and $L$, it is contained in $H(A)$.
Similarly, the lines  $\langle P,Q\rangle$ and $L'$ intersect in the simplex point $X$ and span the plane $\langle P,Q,T\rangle$
which means that $H(X)$ contains $\langle P,Q,T\rangle$.
Therefore, $\langle P,Q,T\rangle$ is contained in \eqref{eq6}.

If $X$ is distinct from $P,Q,A$, then, by Lemma \ref{lemmaF5}, $I(P),I(Q),I(A),I(X)$ span a $3$-dimensional projective space.
In this case, the projective space \eqref{eq6} is a line which contradicts the fact that 
it contains the plane  $\langle P,Q,T\rangle$.
Therefore, $L'$ intersects  $\langle P,Q\rangle$ in one of the points $P,Q,A$. 
Similarly,  we show that 
$L'$ intersects  $\langle P,T\rangle$ and $\langle Q,T\rangle$ in one of the points $P,T,B$ and $Q,T,C$, respectively.

By Lemma \ref{lemma-4p2}, the simplex points $A,B,C$ are non-collinear to the simplex points $T,Q,P$ (respectively).
Therefore, if $L'$ contains $A$, then it is $\langle P,Q\rangle$ or $L$. 
Similarly, we obtain that $L'$ is one of the lines  \eqref{eq4}  if it contains $P$ or $Q$.
\end{proof}

\begin{prop}\label{prop-ntops}
For every simplex line $L$ there is a one-to-one correspondence between 
tops of $\Gamma$ containing $L$ and triples of mutually distinct points on $L$. 
\end{prop}

\begin{proof}
Consider a top of $\Gamma$ containing $L$.
Let $L_i$, $i\in \{1,2,3\}$ be the remaining $3$ lines belonging to this top.
The line $L_i$ intersects $L$ in the point denoted by $P_i$.
The points $P_{i}$, $i\in \{1,2,3\}$ are mutually distinct (Fig.2).
Since $I(P_1),I(P_2),I(P_3)$ span a projective plane (Proposition \ref{propF}),
\begin{equation}\label{eq7}
H(P_1)\cap H(P_2)\cap H(P_3)
\end{equation}
is a projective plane.
For any (not necessarily distinct) $i,j\in\{1,2,3\}$
the line $L_i$ contains the point $P_j$ or it contains two points collinear to $P_j$ (see Fig.2)
and, consequently, $L_i$ is contained in $H(P_j)$ by Lemma \ref{lemmaH}.
Therefore, every line of the top is contained in the plane \eqref{eq7}, i.e. the top is completely determined by the points $P_{i}$, $i\in \{1,2,3\}$.
\begin{center}
\begin{tikzpicture} 

\draw[fill=black] (90:2cm) circle (2pt);
\draw[fill=black] (-30:2cm) circle (2pt);
\draw[fill=black] (210:2cm) circle (2pt);
\draw (0,0) circle (1cm);

\draw (90:2cm)--(-30:2cm)--(210:2cm)--cycle;

\draw[fill=black] (-90:1cm) circle (1.5pt);
\draw[fill=black] (30:1cm) circle (1.5pt);
\draw[fill=black] (150:1cm) circle (1.5pt);

\node at (-2.1cm,-1cm) {$P_1$};
\node at (2.1cm,-1cm) {$P_2$};
\node at (0,-1.3cm) {$P_3$};

\node at (30:-0.75cm) {$L_3$};

\node at (115:1.6cm) {$L_1$};
\node at (60:1.6cm) {$L_2$};
\node at (-50:1.6cm) {$L$};

\end{tikzpicture}
\captionof{figure}{ }
\end{center}

Let $P,Q,T$ be three mutually distinct points on $L$.
Suppose that $$I(\langle I(P),I(Q)\rangle)\setminus \{P,Q\}=\{N_1,N_2,N_3,N_4\}.$$
Then for every $i\in \{1,2,3,4\}$ the projective plane spanned by the line $L$ and the point $N_i$ defines a top containing $L$.
The remaining $3$  lines belonging to this top intersect $L$ in $P,Q$ and a third point denoted by $T_i$.
Since $P,Q,N_i,N_j$ ($i\ne j$) span a $3$-dimensional projective space (Lemma \ref{lemmaF5}), 
the tops corresponding to distinct $N_i$ and $N_j$ are distinct. 
It was established above that the plane spanned by $L$ and $N_i$ is
$$H(P)\cap H(Q)\cap H(T_i).$$
For distinct $i,j\in \{1,2,3,4\}$ the points $T_i,T_j$ are distinct 
(otherwise, distinct $N_i$ and $N_j$ defines the same top). 
This means that $T=T_i$ for a certain $i\in \{1,2,3,4\}$.
\end{proof}

Lemma \ref{lemma3p} and Proposition \ref{prop-ntops} show that for any two tops of $\Gamma$
there is a monomial linear automorphism of $V$ transferring one of these tops to the other.
Also,  Proposition \ref{prop-ntops} implies that every simplex line is contained in precisely $\binom{6}{3}=20$ tops of $\Gamma$.

\section{Remarks on the case $q>5$}
If $P,Q$ are collinear simplex points, then every point from the set
\begin{equation}\label{set}
I(\langle I(P),I(Q)\rangle)\setminus\{P,Q\}
\end{equation}
is collinear to both $P,Q$ and does not belong to the line $\langle P,Q\rangle$ (Proposition \ref{propF}). 
In the case $q=5$, this is the set of all simplex points which are collinear to both $P,Q$ and not on the line $\langle P,Q\rangle$.
The below examples show that for $q=7,11$ this fails, i.e.  there are simplex points which are collinear to both $P,Q$
and do not belong to  \eqref{set} and  $\langle P,Q\rangle$.
Furthermore, for $q=11$ there are tops of $\Gamma$ containing precisely $3$ lines.

Using computing (Scilab) we obtain the following examples.

\begin{exmp}{\rm
Let $q=7$. Then $n=8$. Consider the vectors
$$x=(0, 1, 1, 1, 1, 1, 1, 1)\;,y=(1, 0, 1, 3, 2, 6, 4, 5),$$
$$z=(1, 2, 3, 4, 1, 5, 6, 0)\;,z'=(1, 2, 6, 1, 4, 3, 5, 0).$$
Then $P=\langle x\rangle$ and  $Q=\langle y\rangle$ are collinear simplex points.
The simplex points $T=\langle z \rangle$ and $T'=\langle z'\rangle$ 
are collinear to both $P,Q$ and not on the line $\langle P,Q\rangle$.
The set \eqref{set} contains $T$ and does not contain $T'$.
The planes $\langle P,Q,T\rangle$ and $\langle P,Q,T'\rangle$
define distinct tops of $\Gamma$. Each of these tops contains precisely $4$ lines. 
Three lines of the first top are spanned by pairs of distinct $X,Y\in \{P,Q,T\}$
and the fourth line intersects $\langle P,Q\rangle$ in the point
$$R=\langle 5x+y \rangle=\langle 1, 5, 6, 1, 0, 4, 2, 3\rangle.$$
Similarly, three lines of the second top are spanned by pairs of distinct $X,Y\in \{P,Q,T'\}$
and the fourth line intersects $\langle P,Q\rangle$ in the same point $R$.
So, $\langle P,Q\rangle$ belongs to the both tops and the remaining lines of the tops  intersect this line  in
$P,Q,R$. Note that for $q=5$ such possibility is not realized (Proposition \ref{prop-ntops}).
}\end{exmp}

\begin{exmp}{\rm
Let $q=11$. Then $n=12$. Consider the collinear simplex points
$$P=\langle 0, 1, 1, 1, 1, 1, 1, 1, 1, 1, 1, 1\rangle,\; Q=\langle 1, 0, 1, 2, 4, 8, 5, 10, 9, 7, 3, 6\rangle.$$
Each of the simplex points 
$$T=\langle 1, 10, 0, 9, 6, 2, 7, 5, 3, 8, 4, 1\rangle,\; T'=\langle 1, 10, 0, 3, 2, 9, 6, 7, 5, 8, 4, 1\rangle$$
is collinear to both $P,Q$ and not on the line $\langle P,Q\rangle$.
As in the previous example, \eqref{set} contains $T$ and does not contain $T'$.
The planes $\langle P,Q,T\rangle$ and $\langle P,Q,T'\rangle$
define distinct tops of $\Gamma$. In contrast to the previous example, each of these tops contains precisely $3$ lines. 
}\end{exmp}

\end{document}